\documentclass[11pt]{amsart}
\usepackage{mathrsfs,amsmath,amssymb,amsthm}
\usepackage[dvipdfm]{graphicx}
\newtheorem{theorem}{Theorem}

\newtheorem{Example}{Example}
\newtheorem{proposition}{Proposition}
\newtheorem{lemma}{Lemma}

\newtheorem{remark}{Remark}

\newcommand{\zz}{\mathbf z}

\newcommand{\ww}{\mathbf w}

\title{Topology of the Milnor fibrations of polar weighted homogeneous polynomials}
\author{Kazumasa Inaba} 
\address{Mathematical Institute, Tohoku University, Sendai 980-8578, Japan}
\email{sb0d02@math.tohoku.ac.jp}

\begin{document}
\renewcommand{\thefootnote}{\fnsymbol{footnote}}
\footnote[0]{2010\textit{ Mathematics Subject Classification}.
Primary 32S55; Secondary: 57M25, 32S30, 32S40.}


\maketitle


\begin{abstract}
Let $P$ be a $2$-variable polar weighted homogeneous polynomial 
and let $F_t$ be a deformation of $P$ which is also a polar weighted homogeneous polynomial. 
If $\lvert F_{t}\rvert$ is a Morse function on the orbit space of the $S^1$-action, 
then the handle decomposition obtained by this Morse function induces 
a round handle decomposition of the Milnor fibration of $F_t$. 
In the present paper, we describe a round handle decomposition of the Milnor fibration of $F_t$ concretely and 
give the number of round handles by 
the number of positive and negative components of the links of singularities appearing before and after the deformation. 
We also give a formula of characteristic polynomials of these singularities 
by using the decomposition of the monodromy of the Milnor fibration induced by a round handle decomposition. 
\end{abstract}


\section{Introduction}
We consider a polynomial 
of complex variables $\zz = (z_1, \dots, z_n)$ 
which is given by 
\[
P(\zz, \bar{\zz}) := \sum_{i = 1}^{m} c_{i}\zz^{\nu_{i}}\bar{\zz}^{\mu_{i}}, 
\]
where $c_{i} \in \Bbb{C}^*$, 
$\zz^{\nu_{i}} = z^{\nu_{i,1}}_1 \cdots z^{\nu_{i,n}}_n$ and 
$\bar{\zz}^{\mu_{i}} = \bar{z}_{1}^{\mu_{i,1}} \cdots \bar{z}_{n}^{\mu_{i,n}}$ 
for $\nu_{i} = (\nu_{i,1}, \dots, \nu_{i,n})$ and 
$\mu_{i} = (\mu_{i,1}, \dots, \mu_{i,n})$ respectively.  
Here $\bar{\zz}^{\mu_{i}}$ 
represents the complex conjugate of $\zz^{\mu_{i}} = z_{1}^{\mu_{i,1}} \cdots z_{n}^{\mu_{i,n}}$. 
We call $P(\zz, \bar{\zz})$ \textit{a mixed polynomial} of complex variables $\zz = (z_1, \dots, z_n)$. 
A point $\ww \in \Bbb{C}^{n}$ is called \textit{a mixed singular point of $P(\zz, \bar{\zz})$} 
if~the gradient vectors of $\Re P$ and $\Im P$ are linearly dependent at~$\ww$ \cite{O1, O2}. 
Suppose that $P(0,\dots, 0) = 0$ and $P$ has an isolated singularity at the origin. 
There exist positive real numbers $\varepsilon$ and $\delta$ with $\delta \ll \varepsilon \ll 1$ such that 
the map 
\[
P : D^{2n}_{\varepsilon} \cap P^{-1}(\partial D^{2}_{\delta}) \rightarrow \partial D^{2}_{\delta} 
\]
is a locally trivial fibration over $\partial D^{2}_{\delta}$, where 
$D^{2n}_{\varepsilon} = \{ \zz \in \Bbb{C}^{n} \mid \| \zz \| \leq \varepsilon \}$ and 
$D^{2}_{\delta} = \{ \eta \in \Bbb{C} \mid \lvert \eta \rvert \leq \delta \}$.  
This map is called \textit{the Milnor fibration of $P$ at the origin}. 

Ruas, Seade and Verjovsky \cite{RSV} and Cisneros-Molina \cite{C} introduced 
the following classes of mixed polynomials. 
Let $p_{1}, \dots, p_{n}$ and $q_{1}, \dots, q_{n}$ be integers such that 
$\gcd(p_1, \dots, p_n) = \gcd(q_1,\dots,q_n) = 1$. 
We define the~$S^1$-action and the~$\Bbb{R}^{*}$-action on~$\Bbb{C}^{n}$ as follows: 
\begin{equation*}
\begin{split}
s \circ \zz = (s^{p_{1}}z_{1}, \dots, s^{p_{n}}z_{n}), \ \ \ 
r \circ \zz = (r^{q_1}z_{1}, \dots, r^{q_n}z_{n}), 
\end{split}
\end{equation*}
where $s \in S^{1}$ and $r \in \Bbb{R}^{*}$. 
A mixed polynomial $P(\zz, \bar{\zz})$ is called \textit{a polar weighted homogeneous polynomial} 
if there exists a positive integer $d_p$ such that $P(\zz, \bar{\zz})$ satisfies 
\[
P(s^{p_{1}}z_{1}, \dots, s^{p_{n}}z_{n}, \bar{s}^{p_1}\bar{z}_{1}, \dots, \bar{s}^{p_1}\bar{z}_{n})
 = s^{d_p}P(\zz, \bar{\zz}), \ \ s \in S^{1}.  \notag \\
\]
The weight vector $(p_{1}, \dots, p_{n})$ is called \textit{the polar weight} and $d_p$ is called 
\textit{the polar degree}. 
Similarly $P(\zz, \bar{\zz})$ is called \textit{a~radial weighted homogeneous polynomial} 
if there exists a positive integer $d_{r}$ such that 
\[
P(r^{q_1}z_{1}, \dots, r^{q_n}z_{n}, r^{q_1}\bar{z}_{1}, \dots, r^{q_n}\bar{z}_{n}) 
= r^{d_{r}}P(\zz, \bar{\zz}), \ \ r \in \Bbb{R}^{*}. 
\]
The integer $d_r$ is called \textit{the radial degree}. 
If $P$ is  polar and radial weighted homogeneous,  
$P$ admits the global Milnor fibration $P : \Bbb{C}^{n} \setminus P^{-1}(0) \rightarrow \Bbb{C}^{*}$ 
and the monodromy of the Milnor fibration is given by the $S^1$-action, see for instance \cite{RSV, C, O1, O2}.

We study the topology of the Milnor fibration of a mixed polynomial $P$ by using a deformation of $P$. 
Here \textit{a deformation of $P$} is a~polynomial map 
$F : \Bbb{C}^{n} \times \Bbb{R} \rightarrow \Bbb{C}, (\zz, t) \mapsto F_{t}(\zz)$, with $F_{0}(\zz) = P(\zz, \bar{\zz})$. 
A deformation of $P$ is useful to study the Milnor fibration of $P$. 
For complex isolated singularities, 
it is known that there exist a~neighborhood $U$ of the origin of $\Bbb{C}^{n}$ and 
a~deformation $F_{t}$ of a complex polynomial $P(\zz)$ 
such that $F_{t}(\zz)$ is also a~complex polynomial and any singularity of $F_{t}(\zz)$ in $U$ 
is a Morse singularity for any $0 < t \ll 1$, see for instance \cite[Chapter 4]{Lg}. 
A sufficiently small compact neighborhood of each Morse singularity can be regarded as a $2n$-dimensional $n$-handle and 
thus we have a decomposition 
\[
D^{2n}_{\varepsilon} \cap F_{t}^{-1}(D^{2}_{\delta}) \cong 
\bigr(D^{2n}_{\varepsilon} \cap F_{t}^{-1}(D^{2}_{\delta_{t}})\bigl) \cup_{\varphi} \bigr(\sqcup_{i=1}^{\ell} (\text{$n$-handle})_{i}\bigl), 
\]
where $\ell$ is the Milnor number of the singularity of $P$ at $(0, \dots, 0)$, 
$\varphi = (\varphi_{1}, \dots \varphi_{\ell})$ is the $\ell$-tuple of the attaching map $\varphi_{i}$ of 
$(n\text{-handle})_i$ 
and $D^{2}_{\delta_{t}}$ is a $2$-disk centered at $0$ with radius $\delta_{t}$ 
such that $\delta_{t} < \delta$ and $F_{t} \mid_{F_{t}^{-1}(D^{2}_{\delta_{t}})}$ has no singularities. 
Note that the framing of each handle attaching is determined by the vanishing cycle of the corresponding Morse singularity \cite{K}. 
Such a decomposition induces a decomposition of the monodromy of the Milnor fibration into those of 
the Morse singularities. 

In this paper, we observe analogous deformations for mixed singularities. 
Let $P$ be a $2$-variable polar and radial weighted homogeneous 
polynomial which has an isolated singularity at the origin of $\Bbb{C}^{2}$ 
and let $F_t$ be a deformation of $P$. 
Set $S_{k}(F_{t}) = \{ \zz \in U \mid \text{rank}\>dF_{t}(\zz) = 2-k \}$ for $k = 1, 2$. 
We assume that $F_t$ satisfies the following properties: 
\begin{enumerate}
\item
$F_t$ is polar weighted homogeneous for any $0 \leq t \ll 1$, which implies that, for each $0 < t \ll 1$, 
the singular set $S_{1}(F_{t}) \cup S_{2}(F_{t})$ 
consists of the union of a finite number of orbits of the $S^1$-action \cite[Proposition 2]{In3} and 
$F_{t}(S_{1}(F_{t}))$ consists of circles centered at $0$ except the origin; 
\item
For each point $\ww \in S_{1}(F_{t})$, 
there exist local coordinates $(x_{1}, x_{2}, x_{3}, x_{4})$ centered at $\ww$  
such that $F_{t}$ is given by 
\[
(F_{t}/\lvert F_{t} \rvert, \lvert F_{t} \rvert) = \bigr(x_{1}, -x_{2}^{2}+x_{3}^{2}+x_{4}^{2}+c_{t, \ww}\bigl), 
\]
where $c_{t,\ww} = \lvert F_{t}(\ww)\rvert$ for $\ww \in S_{1}(F_{t})$ and $0 < t \ll 1$. 
In particular, $S_{1}(F_{t})$ consists of indefinite fold singularities; 
\item
$S_{2}(F_{t}) = \{ \mathbf{o}\}$ or $\emptyset$. 
\end{enumerate}
In \cite{In3}, we focused on the mixed singularity of $f\bar{g}$, 
where $f$ and $g$ are $2$-variable weighted homogeneous complex polynomials which have no common branches, 
and showed that there exists a deformation $F_t$ of $f\bar{g}$ 
such that $F_t$ satisfies the above conditions. 
It is important to notice that there exist polar weighted homogeneous polynomials which do not admit 
deformations into smooth maps with only Morse singularities, see \cite[Theorem 1]{In1}, \cite[Corollary 1 and Corollary 2]{In2}.

By the condition $(2)$, the absolute value $\lvert F_{t}\rvert$ of $F_t$ defines a Morse function on 
the orbit space $(D^{4}_{\varepsilon} \cap F_{t}^{-1}(D^{2}_{\delta}))/S^{1}$ of the $S^1$-action for $t>0$ and 
outside the image of the origin, and the indices of the Morse singularities are always $1$. 
Then the handle decomposition of the orbit space according to this Morse function induces a decomposition 
of $D^{4}_{\varepsilon} \cap F_{t}^{-1}(D^{2}_{\delta})$ into a tubular neighborhood of a singular fiber over 
the origin and a finite number of round $1$-handles, that is, we have 
\begin{equation}
\label{0.1}
\tag{0.1}
D^{4}_{\varepsilon} \cap F_{t}^{-1}(D^{2}_{\delta}) \cong 
\bigr(D^{4}_{\varepsilon} \cap F_{t}^{-1}(D^{2}_{\delta_{t}})\bigl) \cup_{\varphi} 
\bigr(\sqcup_{i=1}^{\ell} (\text{round $1$-handle})_{i}\bigl) 
\end{equation}
where $\ell$ is the number of the singularities of the Morse function on the orbit space outside the origin 
and $\varphi = (\varphi_{1}, \dots \varphi_{\ell})$ is the attaching map of $\ell$ copies of a 
round $1$-handle. 
Here we may assume that the images of $\varphi_{1}, \dots, \varphi_{\ell}$ in 
$\partial(D^{4}_{\varepsilon} \cap F_{t}^{-1}(D^{2}_{\delta_{t}}))$ are disjoint. 

In this paper, we describe the structure of this decomposition more precisely. To state our theorem, 
we introduce the notion of negative link components. 
Let $P$ be a polar weighted homogeneous polynomial. Then the link of $P$ at the origin is a Seifert link, 
denoted by $L(P, \mathbf{o})$. 
A fiber surface of the Seifert link induces an orientation to each link component canonically. 
A connected component of $L(P, \mathbf{o})$ is called a \textit{positive component} if the orientation of the link 
component coincides with that of the $S^1$-action, and otherwise it is called a \textit{negative component}. 
Let $\lvert L^{+}(P, \mathbf{o}) \rvert$ and $\lvert L^{-}(P, \mathbf{o}) \rvert$ 
denote the number of positive components of $L(P, \mathbf{o})$ and 
the number of negative components of $L(P, \mathbf{o})$, respectively. 
Then the decomposition is given as follows:

\begin{theorem}\label{rthm}
Let $P$ be a $2$-variable polar and radial weighted homogeneous 
polynomial which has an isolated singularity at the origin and 
let $F_t$ be a deformation of $P$ satisfying the conditions $(1), (2)$ and $(3)$. 
Then 
\renewcommand{\theenumi}{\roman{enumi}}
\begin{enumerate}
\item
$D^{4}_{\varepsilon} \cap F_{t}^{-1}(D^{2}_{\delta_{t}})$ is diffeomorphic to 
the disjoint union of a $4$-ball and $\ell$ copies of $S^{1} \times B^{3}$, 
where $B^3$ is a $3$-ball, and each $\varphi_{i}$ of the attaching map 
$\varphi = (\varphi_{1}, \dots, \varphi_{\ell})$ 
maps the two attaching regions of the $i$-th round $1$-handle to 
both of the boundary of the $4$-ball and that of the $i$-th $S^1 \times B^3$, 
and these $\ell+1$ connected components are connected by $\ell$ round $1$-handles attached by the map $\varphi$; and
\item
the number $\ell$ of round $1$-handles in the decomposition $\eqref{0.1}$ is given as 
\[
\ell = \lvert L^{+}(P, \mathbf{o})\rvert - \lvert L^{+}(F_{t}, \mathbf{o})\rvert = \lvert L^{-}(P, \mathbf{o})\rvert - \lvert L^{-}(F_{t}, \mathbf{o})\rvert 
\]
for $0 < t \ll 1$. 
\end{enumerate}
\end{theorem}
As we mentioned, in \cite{In3}, a deformation of a mixed singularity of type $(f\bar{g}, \mathbf{o})$ is given explicitly. 
In that case, the number $\ell$ is determined by the polar degree and the radial degree of $P$ as 
$\ell = \frac{d_{r} - d_{p}}{2pq}$. 
From the decomposition in $\eqref{0.1}$, we can observe that the Milnor fiber of $(P, \mathbf{o})$ is obtained from 
the Milnor fiber of $(F_{t}, \mathbf{o})$ by removing $2d_{p}$ disks from 
two connected components of $D^{4}_{\varepsilon} \cap F_{t}^{-1}(D^{2}_{\delta_{t}})$ 
and gluing these boundary circles by $d_p$ annuli for each $i = 1, \dots \ell$. 
Moreover, we see that monodromy exchanges these $\ell$ copies of the union of $d_p$ annuli 
and $2d_{p}$ disks cyclically. 

This paper is organized as follows. In Section $2$ 
we give the definitions of fold singularities and round handles and 
introduce deformations of polar weighted homogeneous polynomials. 
In Section $3$ we prove Theorem $1$. In Section $4$ we make a few comments on the monodromy of the Milnor fibration of $F_t$ 
and a specific deformation of $f\bar{g}$. 

The author would like to thank Professor Masaharu Ishikawa for precious comments.

\section{Preliminaries}
\subsection{Fold singularities}
Let $X$ be a $n$-dimensional manifold and $W$ be a $2$-dimensional manifold. 
We denote $C^{\infty}(X, W)$ the set of smooth maps from $X$ to $W$. 
It is known that the subset of smooth maps from $X$ to $W$ whose singularities are 
only definite fold singularities, indefinite fold singularities or cusps 
is open and dense in $C^{\infty}(X, W)$ topologized with 
the $C^{\infty}$-topology \cite{L1}. 
Here \textit{a fold singularity} is the singularity of the following map 
\[
(x_{1}, \dots, x_{n}) \mapsto (x_{1}, \sum_{j=2}^{n}\pm x_{j}^{2}), 
\] 
where $(x_{1}, \dots, x_{n})$ are coordinates of a small neighborhood of the singularity. 
If the coefficients of $x_j$ for $j = 2, \dots, n$ is 
either all positive or all negative, we say that 
$x$ is a \textit{definite fold} singularity, 
otherwise it is an \textit{indefinite fold} singularity.

\subsection{Deformations of polar weighted homogeneous polynomials}
Let $P : \Bbb{C}^{2} \rightarrow \Bbb{C}$ be a polar weighted homogeneous polynomial map which has an isolated singularity at the origin of $\Bbb{C}^2$. 
Then $P$ admits a Milnor fibration, i.e., there exist positive real numbers $\varepsilon$ and $\delta$ such that 
the map 
\[
P : D^{4}_{\varepsilon} \cap P^{-1}(\partial D^{2}_{\delta}) \rightarrow \partial D^{2}_{\delta} 
\]
is a locally trivial fibration over $\partial D^{2}_{\delta}$, where 
$D^{4}_{\varepsilon} = \{ \zz \in \Bbb{C}^{2} \mid \| \zz \| \leq \varepsilon \}$ and 
$D^{2}_{\delta} = \{ \eta \in \Bbb{C} \mid \lvert \eta \rvert \leq \delta \}$. 
We fix such  positive real numbers $\varepsilon$ and $\delta$. 
Let $F$ be the fiber surface of a polar weighted homogeneous polynomial $P$. 
In \cite{O1, O2}, the monodromy $h : F \rightarrow F$ of the Milnor fibration of $P$ is given by 
\[
(z_{1}, z_{2}) \mapsto \Bigr(\exp\Bigr(\frac{2p\pi i}{d_p}\Bigl)z_{1}, \exp\Bigr(\frac{2q\pi i}{d_p}\Bigl)z_{2}\Bigl), 
\] 
where $(p,q)$ is the polar weight of $P$. 
Note that the link $K_{P} = \partial D^{4}_{\varepsilon} \cap P^{-1}(0)$ is an invariant set of the $S^1$-action. 
So the link $K_{P}$ is a Seifert link in the $3$-sphere \cite{EN}. 

Let $F_t$ be a deformation of $P$ which satisfies the conditions $(1), (2)$ and $(3)$. 
Since the fiber surface $F_{0}^{-1}(c)$ intersects $\partial D^{4}_{\varepsilon}$ transversely, 
$F_{t}^{-1}(c)$ intersects $\partial D^{4}_{\varepsilon}$ transversely for each $c \in D^{2}_{\delta}$ 
and $0 \leq t \ll 1$. 
By the Ehresmann's fibration theorem \cite{W}, the map 
\[
F_{t} : D^{4}_{\varepsilon} \cap F_{t}^{-1}(\partial D^{2}_{\delta}) \rightarrow \partial D^{2}_{\delta}
\] 
is a locally trivial fibration for $0 \leq t \ll 1$. 
The polar weight of $F_t$ coincides with that of $F_{0}$ for $0 < t \ll 1$. 
Thus the monodromy of $F_{t} : D^{4}_{\varepsilon} \cap F_{t}^{-1}(\partial D^{2}_{\delta}) \rightarrow \partial D^{2}_{\delta}$ 
is given by the same $S^1$-action on $\Bbb{C}^{2}$ for each $0 \leq t \ll 1$. 

\begin{lemma}
The Milnor fibration 
$F_{t} : D^{4}_{\varepsilon} \cap F_{t}^{-1}(\partial D^{2}_{\delta}) \rightarrow \partial D^{2}_{\delta}$ 
is isomorphic to the fibration 
$F_{0}/\lvert F_{0} \rvert : \partial D^{4}_{\varepsilon} \setminus \emph{Int}N(K_{F_{0}}) \rightarrow S^{1}$ 
for $0 \leq t \ll 1$, 
where $N(K_{F_{0}}) = \{\zz \in \partial D^{4}_{\varepsilon} \mid \lvert F_{0}(\zz)\rvert \leq \delta\}$. 
\end{lemma}
\begin{proof}
Since the fiber surface $F_{0}^{-1}(c)$ is transversal to $\partial D^{4}_{\varepsilon}$, 
$F_{t}^{-1}(c)$ is transversal to $\partial D^{4}_{\varepsilon}$ for any $c \in D^{2}_{\delta}$ 
and $0 \leq t \ll 1$. 
Fix such a positive real number $t$. 
We set 
\[\begin{split}
\partial \mathcal{E}(\delta, \varepsilon)& := \{(\zz, t') \in \Bbb{C}^{2} \times [0, t] \mid \lvert F_{t'}(\zz)\rvert = \delta, \| \zz \| \leq \varepsilon \}\\
\partial ^2\mathcal{E}(\delta, \varepsilon)& := \{(\zz, t') \in \Bbb{C}^{2}\times [0, t] \mid \lvert F_{t'}(\zz)\rvert = \delta, \| \zz \| = \varepsilon \}.
\end{split}
\]
Then the projection 
\[
\pi' :( \partial \mathcal{E}(\delta, \varepsilon),\partial^2\mathcal{E}(\delta, \varepsilon)) \rightarrow [0, t], \ \ 
(\zz, t') \mapsto t'
\]
is a~proper submersion. 
By the Ehresmann's fibration theorem \cite{W}, 
$\pi'$ is a~locally trivial fibration over $[0, t]$. 
Thus the projection 
$\pi'$ induces 
a~family of isomorphisms $\psi_{t'} : \partial E_{0}(\delta, \varepsilon) \rightarrow \partial E_{t'}(\delta, \varepsilon)$ 
of fibrations 
such that the following diagram is commutative for $0 \leq t' \leq t$: 
\def\mapright#1{\smash{\mathop{\longrightarrow}\limits^{{#1}}}}
\def\mapdown#1{\Big\downarrow\rlap{$\vcenter{\hbox{$#1$}}$}}
\[\begin{matrix}
\partial E_{0}(\delta, \varepsilon)&\mapright{\psi_{t'}}& \partial E_{t'}(\delta, \varepsilon) \\
\mapdown{F_{0}}&&\mapdown{F_{t'}}\\
S^{1}_{\delta}&=& S^{1}_{\delta} 
\end{matrix},\]
where $\partial E_{t'}(\delta, \varepsilon) = 
\{\zz \in \Bbb{C}^{2} \mid \lvert F_{t'}(\zz)\rvert = \delta, \| \zz \| \leq \varepsilon \}$. 
Thus the two fibrations 
$F_{t'} : D^{4}_{\varepsilon} \cap F_{t'}^{-1}(\partial D^{2}_{\delta}) \rightarrow \partial D^{2}_{\delta}$ and 
$F_{0} : D^{4}_{\varepsilon} \cap F_{0}^{-1}(\partial D^{2}_{\delta}) \rightarrow \partial D^{2}_{\delta}$ 
are isomorphic for $0 \leq t' \leq t$. 
By \cite[Theorem 37]{O2}, the two fibrations 
\[
F_{0} : \partial E_{0}(\delta, \varepsilon) \rightarrow S^{1}_{\delta}, \ \ \
F_{0}/\lvert F_{0}\rvert : S^{3}_{\varepsilon} \setminus \text{Int}N(K_{F_{0}}) \rightarrow S^{1} 
\]
are isomorphic for $\varepsilon > 0, \delta > 0$ and $\delta \ll \varepsilon$. 
Thus $F_{t} : D^{4}_{\varepsilon} \cap F_{t}^{-1}(\partial D^{2}_{\delta}) \rightarrow \partial D^{2}_{\delta}$ 
is isomorphic to 
$F_{0}/\lvert F_{0} \rvert : \partial D^{4}_{\varepsilon} \setminus \text{Int}N(K_{F_{0}}) \rightarrow S^{1}$ 
for $0 \leq t \ll 1$. 
\end{proof}

\begin{lemma}\label{c1}
The orbit space of  
$D^{4}_{\varepsilon} \cap F_{t}^{-1}(\partial D^{2}_{\delta})$ 
of the $S^1$-action is homeomorphic to 
a holed $2$-sphere for $0 \leq t \ll 1$. 
\end{lemma}
\begin{proof}
The monodromy of $F_{t} : D^{4}_{\varepsilon} \cap F_{t}^{-1}(\partial D^{2}_{\delta}) \rightarrow \partial D^{2}_{\delta}$ 
is given by the same $S^1$-action on $\Bbb{C}^{2}$ for each $0 \leq t \ll 1$. 
By Lemma $1$, $(D^{4}_{\varepsilon} \cap F_{t}^{-1}(\partial D^{2}_{\delta}))/S^{1}$ is homeomorphic to 
$(\partial D^{4}_{\varepsilon} \setminus \text{Int}N(K_{F_{0}}))/S^{1}$.

Since the orbit space $\partial D^{4}_{\varepsilon}/S^{1}$ is homeomorphic to a $2$-sphere and 
$K_{F_{0}}$ is an invariant set of the $S^1$-action, 
the orbit space $(D^{4}_{\varepsilon} \cap F_{t}^{-1}(\partial D^{2}_{\delta}))/S^{1}$ is a holed $2$-sphere.
\end{proof}

\subsection{Round handles}
Let $X$ and $Y$ be $n$-dimensional smooth manifolds. 
According to \cite{As, M}, we say that \textit{$X$ is obtained from $Y$ by attaching a round $k$-handle} if 
\begin{enumerate}
\item there are disk bundles over $S^{1}$, $E_{s}^{k}$ and $E_{u}^{n-k-1}$, \\
\item there exists an embedding $\varphi : \partial E_{s}^{k} \times_{S^{1}} E_{u}^{n-k-1} \rightarrow \partial Y$ 
such that $X \cong Y \cup_{\phi} E_{s}^{k} \oplus E_{u}^{n-k-1}$, 
\end{enumerate}
where $E_{s}^{k} \oplus E_{u}^{n-k-1}$ is the Whitney sum of $E_{s}^{k}$ and $E_{u}^{n-k-1}$ over $S^1$. 
The bundle $E_{s}^{k} \oplus E_{u}^{n-k-1}$ over $S^1$ is called 
\textit{an $n$-dimensional round $k$-handle} and $\varphi$ is called 
\textit{the attaching map of $E_{s}^{k} \oplus E_{u}^{n-k-1}$}. 
Note that a sufficiently small compact neighborhood of a connected component of the set of fold singularities 
can be regarded as an $n$-dimensional round handle. 
In our case, a sufficiently small compact neighborhood of each connected component of $S_{1}(F_{t})$ 
is regarded as a $4$-dimensional round $1$-handle.

\section{Proof of Theorem \ref{rthm}}
\subsection{Round $1$-handles determined by $S_{1}(F_{t})$}
By the condition $(3)$, the origin $\mathbf{o}$ is an isolated singularity of $F_t$. 
There exist positive real numbers 
$\varepsilon_{t}$ and $\delta_{t}$ such that $\delta_{t} \ll \varepsilon_{t}$ and the map 
\[
F_{t} : D^{4}_{\varepsilon'_{t}} \cap F_{t}^{-1}(\partial D^{2}_{\delta'_{t}}) \rightarrow \partial D^{2}_{\delta'_{t}} 
\]
is a locally trivial fibration over $\partial D^{2}_{\delta'_{t}}$, where 
$D^{4}_{\varepsilon'_{t}} = \{ \zz \in \Bbb{C}^{2} \mid \| \zz \| \leq \varepsilon'_{t} \}$ and 
$D^{2}_{\delta'_{t}} = \{ \eta \in \Bbb{C} \mid \lvert \eta \rvert \leq \delta'_{t} \}$ 
for $0 < \varepsilon'_{t} \leq \varepsilon_{t}, 0 < \delta'_{t} \leq \delta_{t}$ and 
$\delta'_{t} \ll \varepsilon'_{t}$, see \cite{W} . 
Thus $F_{t}^{-1}(c)$ intersects $\partial D^{4}_{\varepsilon_{t}}$ transversely for any $c \in D^{2}_{\delta_{t}}$ 
and $0 \leq t \ll 1$. 
We assume that $\varepsilon_{t}$ and $\delta_{t}$ satisfy the following properties: 
\[
D^{4}_{\varepsilon_{t}} \cap S_{1}(F_{t}) = \emptyset, \ \ \ 
D^{2}_{\delta_{t}} \cap F_{t}(S_{1}(F_{t})) = \emptyset. 
\]
See Figure~\ref{fig1}.
\begin{figure}[h]
\centering
  \includegraphics[scale=1.1]{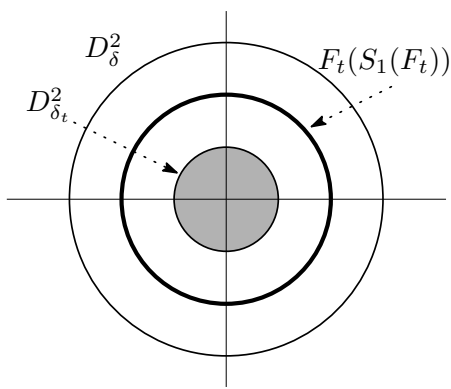}
  \caption{$D^{2}_{\delta_{t}}$ and $F_{t}(S_{1}(F_{t}))$}
  \label{fig1}
\end{figure}
Put $M_{0} = D^{4}_{\varepsilon} \cap F_{t}^{-1}(D^{2}_{\delta_{t}})$.

Fix $t$ with $0 < t \ll 1$ and 
let $\ell$ be the number of singularities of $\lvert F_{t} \rvert$. 
Note that $\lvert F_{t} \rvert$ is a Morse function by the conditions $(1)$ and $(2)$ except the origin. 
Let $C_{1}, \dots, C_{\ell}$ be the connected components of $S_{1}(F_{t})$, where 
the number of the connected components of $S_{1}(F_{t})$ is $\ell$ because of the conditions $(1), (2)$ and $(3)$. 
We may assume that $\lvert F_{t} \rvert$ satisfies 
\[
\lvert F_{t}(c_{1}) \rvert \leq \lvert F_{t}(c_{2}) \rvert \dots \leq \lvert F_{t}(c_{\ell}) \rvert
\]
for $c_{i} \in C_{i}$ and $i = 1, \dots, \ell$. 
Let $c'_{i}$ be the singularity of $\lvert F_{t} \rvert$ corresponding to $C_{i}$ and 
$N'_{i}$ be a sufficiently small compact neighborhood of $c'_{i}$ for $i = 1, \dots, \ell$. 
Since $\lvert F_{t} \rvert$ is a Morse function, each $N'_{i}, i = 1,\dots, \ell$, 
can be regarded as a $3$-dimensional $1$-handle $[-1, 1] \times D^{2}_{i}$, 
where $D^{2}_{i}$ is a $2$-disk. 
We set $M'_{0} = M_{0}/S^{1}$ and $M'_{i} := M'_{i-1} \cup_{\varphi'_{i}} N'_{i}$, 
where $\varphi'_{i} : \{\pm 1\} \times D^{2}_{i} \rightarrow \partial M'_{i-1}$ is the attaching map of $N'_{i}$ 
for $i = 1,\dots, \ell$. 
We may assume that $\varphi'_{i}(\{\pm 1\} \times D^{2}_{i}) \subset \partial M'_{0}$ for $i = 1,\dots, \ell$. 
Then the orbit space $D^{4}_{\varepsilon}/S^{1}$ is a topological manifold obtained from 
$M_{0}'$ by attaching $3$-dimensional $1$-handles $N'_{1}, \dots, N'_{\ell}$. 
Note that $D^{4}_{\varepsilon}/S^{1}$ is homeomorphic to a $3$-ball.

\begin{lemma}\label{l2}
Let $M_{0}^{*}$ be a connected component of $M'_{0}$. 
Then $\varphi'_{i}(\{\pm 1\} \times D^{2}_{i}) \not\subset \partial M_{0}^{*}$ for $i = 1,\dots, \ell$. 
\end{lemma}
\begin{proof}
Assume that there exist 
$i \in \{ 1, \dots, \ell \}$ and 
a connected component $M_{0}^{*}$ of $M'_{0}$ such that 
$\varphi'_{i}(\{\pm 1\} \times D^{2}_{i})$ is contained in $\partial M_{0}^{*}$. 
Then the genus of $\partial M'_{i}$ is greater than $0$. 
After attaching $1$-handles, the genus of the boundary of the orbit space does not decrease. 
Thus the genus of $\partial M'_{\ell}$ is greater than $0$. 
As Lemma \ref{c1}, the genus of $\partial M'_{\ell}$
is equal to $0$. This is a~contradiction. 
\end{proof}
Let $M_{i}$ and $N_i$ be $4$-dimensional manifolds such that 
$M_{i}/S^{1} = M'_{i}$ and $N_{i}/S^{1} = N'_{i}$ respectively for $i = 1, \dots, \ell$. 
Then $N_i$ can be regarded as a $4$-dimensional round $1$-handle and 
$M_{i}$ is a manifold obtained from $M_{i-1}$ by attaching $N_{i}$ for $i = 1, \dots, \ell$. 
By Lemma \ref{l2}, $N_{i}$ connects two connected components of $M_{0}$. 
Note that $M_{\ell}$ is diffeomorphic to 
$D^{4}_{\varepsilon} \cap F_{t}^{-1}(D^{2}_{\delta})$. 

\subsection{The fiber surface of $F_{t} : D^{4}_{\varepsilon_{t}} \cap F_{t}^{-1}(\partial D^{2}_{\delta_{t}}) \rightarrow \partial D^{2}_{\delta_{t}}$} 
We consider the restricted Milnor fibration  
$F_{t} : D^{4}_{\varepsilon} \cap F_{t}^{-1}(\partial D^{2}_{\delta_{t}}) \rightarrow \partial D^{2}_{\delta_{t}}$ and 
connected components of $M_0$. 

\begin{lemma}\label{l3} 
Let $S_{0}$ be the fiber surface of 
$F_{t} : D^{4}_{\varepsilon} \cap F_{t}^{-1}(\partial D^{2}_{\delta_{t}}) \rightarrow \partial D^{2}_{\delta_{t}}$. 
Then $S_{0}$ 
is diffeomorphic to the disjoint union of the fiber surface of 
$F_{t} \mid_{D^{4}_{\varepsilon_{t}} \cap F_{t}^{-1}(\partial D^{2}_{\delta_{t}})}$ and  
$\ell$ copies of an annulus, where $\ell$ is the number of connected components of $S_{1}(F_{t})$. 
\end{lemma}
\begin{proof}
Let $M_{0}^{0}, M_{0}^{1}, \dots, M_{0}^{k}$ denote the connected components of $M_{0}$ such that 
$\mathbf{o} \in M_{0}^{0}$. 
Then $M_{0}^{0} \cap D^{4}_{\varepsilon_{t}} \neq \emptyset$. 
The restricted map $F_{t} : D^{4}_{\varepsilon_{t}} \cap F_{t}^{-1}(D^{2}_{\delta_{t}}) \rightarrow D^{2}_{\delta_{t}}$ 
has a~unique singularity at the origin $\mathbf{o}$ of $\Bbb{C}^{2}$.  
By \cite[Lemma 11.3]{Mi}, 
$D^{4}_{\varepsilon_{t}} \cap F_{t}^{-1}(\partial D^{2}_{\delta_{t}})$ 
is homeomorphic to $\partial D^{4}_{\varepsilon_{t}} \setminus \text{Int}N(K_{F_{t}})$, where 
$N(K_{F_{t}}) = \{\zz \in \partial D^{4}_{\varepsilon_{t}} \mid \lvert F_{t}(\zz)\rvert \leq \delta_{t} \}$. 
So any fiber surface of 
$F_{t}: D^{4}_{\varepsilon_{t}} \cap F_{t}^{-1}(\partial D^{2}_{\delta_{t}}) \rightarrow \partial D^{2}_{\delta_{t}}$ 
is connected. 
The boundary of the orbit space 
$(D^{4}_{\varepsilon_{t}} \cap M_{0}^{0})/S^{1}$ is homeomorphic to a~$2$-sphere and  
\[
M_{0}^{j} \cap D^{4}_{\varepsilon_{t}} = \emptyset 
\]
for $j = 1, \dots, k$. 

Let $S_{0}^{0}$ be a fiber surface of $F_{t}\mid_{M_{0}^{0} \cap (D^{4}_{\varepsilon_{t}} \cap F_{t}^{-1}(\partial D^{2}_{\delta_{t}}))}$. 
We divide the surface $S_{0}^{0}$ as follows: 
\[
S_{0}^{0} = (S_{0}^{0} \cap D^{4}_{\varepsilon_{t}}) \cup (S_{0}^{0} \cap (D^{4}_{\varepsilon} \setminus \text{Int$D^{4}_{\varepsilon_{t}}$})). 
\]
Since $F_{t} : (D^{4}_{\varepsilon} \setminus 
\text{Int$D^{4}_{\varepsilon_{t}}$}) \cap F_{t}^{-1}(D^{2}_{\delta_{t}}) \rightarrow D^{2}_{\delta_{t}}$ 
has no singularities and  
$F_{t}^{-1}(c)$ intersects $\partial D^{4}_{\varepsilon} \sqcup \partial D^{4}_{\varepsilon_{t}}$ 
transversely for any $c \in D^{2}_{\delta_{t}}$, 
$F_{t}^{-1}(c) \cap (D^{4}_{\varepsilon} \setminus \text{Int$D^{4}_{\varepsilon_{t}}$})$ 
is diffeomorphic to $F_{t}^{-1}(0) \cap (D^{4}_{\varepsilon} \setminus \text{Int$D^{4}_{\varepsilon_{t}}$})$. 
Note that $F_{t}^{-1}(0)$ is an invariant set of the $S^{1}$-action and 
$F_{t}^{-1}(0)/S^{1}$ is a $1$-dimensional algebraic set. 
The orbit space $F_{t}^{-1}(0)/S^{1}$ is diffeomorphic to $[0, 1]$. 
Thus the connected component of 
$F_{t}^{-1}(c) \cap (D^{4}_{\varepsilon} \setminus \text{Int$D^{4}_{\varepsilon_{t}}$})$ is diffeomorphic to an annulus. 
So any connected component of 
$S_{0}^{0} \cap (D^{4}_{\varepsilon} \setminus \text{Int$D^{4}_{\varepsilon_{t}}$})$ is an annulus. 
Since any fiber of $F_{t}$ intersects $\partial D^{4}_{\varepsilon_{t}}$ transversely, 
$S^{0}_{0} \cap \partial D^{4}_{\varepsilon_{t}}$ consists of circles and 
$S_{0}^{0} \cap (D^{4}_{\varepsilon} \setminus \text{Int$D^{4}_{\varepsilon_{t}}$})$ is diffeomorphic to 
$(S^{0}_{0} \cap \partial D^{4}_{\varepsilon_{t}}) \times [0, 1]$. 
So we have 
\begin{equation*}
\begin{split}
S_{0}^{0} &= (S_{0}^{0} \cap D^{4}_{\varepsilon_{t}}) \cup (S_{0}^{0} \cap (D^{4}_{\varepsilon} \setminus \text{Int$D^{4}_{\varepsilon_{t}}$})) \\ 
&\cong (S_{0}^{0} \cap D^{4}_{\varepsilon_{t}}) \cup ((S^{0}_{0} \cap \partial D^{4}_{\varepsilon_{t}}) \times [0, 1]) \\
&\cong S_{0}^{0} \cap D^{4}_{\varepsilon_{t}}.
\end{split} 
\end{equation*}

We consider $M_{0}^{j}$ for $j = 1, \dots, k$. 
The restricted map $F_{t} : M_{0}^{j} \rightarrow D^{2}_{\delta_{t}}$ has no singularities. 
For any $c \in D^{2}_{\delta_{t}} \setminus \{0\}$ 
and $j = 1, \dots, k$, 
$F_{t}^{-1}(c) \cap M_{0}^{j}$ is diffeomorphic to $F_{t}^{-1}(0) \cap M_{0}^{j}$. 
Since $F_{t}^{-1}(0)$ is an invariant set of the~$S^{1}$-action, 
the orbit space $F_{t}^{-1}(0)/S^{1}$ is a $1$-dimensional algebraic set. 
So $F_{t}^{-1}(0)/S^{1}$ is diffeomorphic to $[0, 1]$ or $S^{1}$. 
Assume that $F_{t}^{-1}(0)/S^{1} = S^{1}$. 
Then $F_{t}^{-1}(c)$ is a torus and 
the orbit space $F_{t}^{-1}(c)/S^{1}$ is also a torus. 
Since the boundary of $M_{\ell}'$ is a~$2$-sphere, 
this is a contradiction. 
Let $S_{0}^{j}$ denote the fiber surface of $F_{t}\mid_{M_{0}^{j}}$. 
Then $S_{0}^{j}/S^{1}$ is diffeomorphic to $[0, 1]$ and 
the fiber surface $S_{0}^{j}$ is diffeomorphic to an~annulus for $j = 1, \dots, k$. 

By Lemma \ref{l2}, each $N_{i}$ connects two connected components of $M_0$. 
Since $M_{\ell}$ is connected, we have $k+1-\ell = 1$. 
Thus the number of connected components of $S_0$ other than that of 
$F_{t} : D^{4}_{\varepsilon_{t}} \cap F_{t}^{-1}(\partial D^{2}_{\delta_{t}})$ 
is equal to $\ell$. 
\end{proof}

\begin{lemma}\label{c2}
The connected component $M_{0}^{0}$ of $M_{0}$ is diffeomorphic to a $4$-ball and 
$M_{0}^{j}$ is diffeomorphic to $S^1 \times B^{3}$, where $B^3$ is a $3$-ball, for $j =1, \dots, \ell$. 
\end{lemma}
\begin{proof}
The two fibrations $F_{t}: D^{4}_{\varepsilon_{t}} \cap F_{t}^{-1}(\partial D^{2}_{\delta'}) \rightarrow \partial D^{2}_{\delta'}$ 
and $F_{t}/\lvert F_{t}\rvert : \partial D^{4}_{\varepsilon_{t}} \setminus (\partial D^{4}_{\varepsilon_{t}} \cap F_{t}^{-1}(0)) \rightarrow S^{1}$  
are isomorphic for any $0 < \delta' \leq \delta_{t}$. 
Thus $M_{0}^{0}$ is diffeomorphic to a $4$-ball. 

The map $F_{t}\mid_{M_{0}^{j}}$ has no singularities for $j \neq 0$. 
Then the fiber surface of 
$F_{t} : M_{0}^{j} \cap F_{t}^{-1}(\partial D^{2}_{\delta'}) \rightarrow \partial D^{2}_{\delta'}$
is diffeomorphic to $S_{0}^{j}$ for $0 < \delta' \leq \delta_{t}$. 
Since the monodromy of $F_{t} : M_{0}^{j} \cap F_{t}^{-1}(\partial D^{2}_{\delta'}) \rightarrow \partial D^{2}_{\delta'}$ 
is given by the $S^1$-action, 
the orbit space $(M_{0}^{j} \cap F_{t}^{-1}(D^{2}_{\delta_{t}} \setminus \{0\}))/S^{1}$ 
is homeomorphic to $(S_{0}^{j}/S^{1}) \times (0, 1]$. 
By Lemma \ref{l3}, $S_{0}^{j}/S^{1}$ is also an annulus. We identify $S_{0}^{j}/S^{1}$ with $S^{1} \times [0,1]$. 
Since $S^{1} \times (0, 1]$ is diffeomorphic to $D^{2}\setminus \{0\}$, 
$(M_{0}^{j} \cap F_{t}^{-1}(D^{2}_{\delta_{t}} \setminus \{0\}))/S^{1}$ is homeomorphic to 
$D^{2} \times [0,1] \setminus (\{0\}\times [0,1])$, where $D^2$ is a $2$-disk centered at $0$. 
Since $F_{t}^{-1}(0)$ is an invariant set of the $S^1$-action, 
the orbit space of $F_{t}^{-1}(0)$ is homeomorphic to $\{0\} \times [0, 1]$.
Thus the orbit space of $M_{0}^{j}$ is homeomorphic to $D^{2} \times [0,1]$. 
The manifold $M_{0}^{j}$ is diffeomorphic to $S^{1} \times B^{3}$. 
\end{proof}

\subsection{The number of connected components of $S_{1}(F_{t})$}
To complete the proof of Theorem~\ref{rthm}, 
it is enough to show the equality in Theorem \ref{rthm} $(\text{ii})$. 
We set $\tilde{M}_{0} = D^{4}_{\varepsilon} \cap F_{t}^{-1}(\partial D^{2}_{\delta_{t}})$ and 
$\tilde{M}_{i} = \tilde{M}_{i-1} \cup_{\varphi_{i}} \partial N_{i}$ for $i = 1, \dots, \ell$. 
Since $F_t$ is polar weighted homogeneous, 
the monodromy of $F_{t}\mid_{\tilde{M}_{i}}$ is given by the $S^{1}$-action on~$\Bbb{C}^{2}$. 
By the condition $(2)$, a fiber of $\lvert F_{t}\rvert : N'_{i} \rightarrow \Bbb{R}$ is as follows: 
\begin{equation*}
\lvert F_{t}\rvert^{-1}(\lvert u\rvert) \cap N'_{i} \cong \begin{cases}                                             
          \text{two $2$-disks} & 0 < c_{t,\ww} - \lvert u\rvert \ll 1 \\                                                         
          \text{an annulus} &  0 < \lvert u\rvert - c_{t,\ww} \ll 1           
          \end{cases}. 
\end{equation*}
Since the polar degree of $F_t$ is equal to $d_p$ and $N_i$ is a neighborhood of an orbit of the $S^1$-action, 
a fiber of $F_{t}: N_i \rightarrow D^{2}_{\delta}$ is a $d_p$-fold cover over a fiber of $\lvert F_{t}\rvert$. 
Thus we have 
\begin{equation*}
F_{t}^{-1}(u) \cap N_{i} \cong \begin{cases}                                             
          (\sqcup_{j=1}^{d_{p}} D^{2}_{1,j}) \sqcup (\sqcup_{j=1}^{d_{p}} D^{2}_{2,j}) & 0 < c_{t,\ww} - \lvert u\rvert \ll 1 \\                                                         
          \sqcup_{j=1}^{d_{p}} A_{j} &  0 < \lvert u\rvert - c_{t,\ww} \ll 1           
          \end{cases}, 
\end{equation*}
where $D^{2}_{k, j}$ is a $2$-disk and $A_j$ is an annulus for $k=1, 2$ and $j = 1, \dots, d_{p}$. 
By Lemma \ref{l2}, we may assume that $\sqcup_{j=1}^{d_{p}} D^{2}_{1,j}$ is contained in a connected component of $M_0$ 
which does not contain $\sqcup_{j=1}^{d_{p}} D^{2}_{2,j}$. 
Let $S_{i}$ be the fiber surface of $F_{t}\mid_{\tilde{M}_{i}}$ for $i = 0, 1, \dots, \ell$. 
Set $A' = \sqcup_{j=1}^{d_{p}}A_{j}, \partial A' = \sqcup_{j=1}^{d_{p}}\partial A_{j}$ and 
$D'_{k} = \sqcup_{j=1}^{d_{p}}D^{2}_{k,j}$ for $k = 1, 2$. 
Then $S_i$ is the surface obtained from $S_{i-1}$ by replacing $D'_{1} \sqcup D'_{2}$ by $A'$. 
The Euler characteristic $\chi(S_{i})$ of $S_i$ is equal to $\chi(S_{i-1}) - 2d_{p}$, where 
$d_p$ is the polar degree of $F_t$ for $i = 1, \dots, \ell$. 
Thus we have 
\[
\chi(S_{\ell}) - \chi(S_{0}) = - 2\ell d_{p}. 
\]

\begin{lemma}\label{c3} 
Let $\ell$ be the number of connected components of $S_{1}(F_{t})$. 
Then $\ell$ is equal to $\lvert L^{+}(P, \mathbf{o})\rvert - \lvert L^{+}(F_{t}, \mathbf{o})\rvert$ and also to 
$\lvert L^{-}(P, \mathbf{o})\rvert - \lvert L^{-}(F_{t}, \mathbf{o})\rvert$. 
\end{lemma}
\begin{proof}
Since the fibration $F_{t}\mid_{\tilde{M}_{\ell}}$ is isomorphic to 
$P : D^{4}_{\varepsilon} \cap P^{-1}(\partial D^{2}_{\delta}) \rightarrow \partial D^{2}_{\delta}$ and 
$L(P, \mathbf{o})$ is the Seifert link in $\partial D^{4}_{\varepsilon}$, 
we have 
\[
\chi(S_{\ell}) = 1 - \{pq(m+n) - p - q\}(m-n), 
\]
where $m = \lvert L^{+}(P, \mathbf{o})\rvert$ and $n = \lvert L^{-}(P, \mathbf{o})\rvert$, see \cite[Theorem 11.1]{EN}. 
By Lemma \ref{l3}, the fiber surface $S_{0}$ of $F_{t}\mid_{\tilde{M}_{0}}$ 
is diffeomorphic to $S^{0}_{0} \sqcup S^{1}_{0} \sqcup \dots \sqcup S^{k}_{0}$ 
and $S^{j}_{0}$ is an annulus for $j \neq 0$. 
The Euler characteristic $\chi(S_{0})$ of $S_0$ is equal to $\chi(S^{0}_{0})$. 
Since the link $L(F_{t}, \mathbf{o})$ is also the Seifert link in a $3$-sphere, 
$\chi(S^{0}_{0})$ is given by 
\[
\chi(S^{0}_{0}) = 1 - \{pq(m'+n') - p -q\}(m'-n'), 
\] 
where $m' = \lvert L^{+}(F_{t}, \mathbf{o})\rvert$ and $n' = \lvert L^{-}(F_{t}, \mathbf{o})\rvert$.
On the other hand, $\chi(S_{\ell})$ is equal to $\chi(S_{0}) - 2\ell d_{p}$. 
The polar degree $d_p$ is equal to $pq(m-n)$ and also to $pq(m'-n')$ \cite{EN}. 
Then we have 
\begin{equation*}
\begin{split}
\chi(S_{\ell}) - \chi(S_{0}) &= - \{pq(m+n) - p - q\}(m-n) + \{pq(m'+n') - p -q\}(m'-n') \\
                             &= -pq(m-n)(m+n) + pq(m'-n')(m'+n') \\
                             &= -d_{p}\{(m+n)-(m'+n')\} \\
                             &= - 2\ell d_{p}. 
\end{split}
\end{equation*}
So $2\ell$ is equal to $m+n - (m'+n')$. Since $m-m'$ is equal to $n-n'$, 
$\ell$ is equal to $m - m'$ and also to $n-n'$. 
\end{proof}

We give an example of Lemma $\ref{c3}$ which is considered in \cite{In3}. 
\begin{Example}
Set $f(\zz) = z_{1}^{m} + z_{2}^{m}$ and $g(\zz) = z_{1} + 2z_{2}$ where $m \geq 3$. 
We consider a deformation $F_{t} = f(\zz)\overline{g(\zz)} + t(z_{1}^{m}\bar{z}_{1} + z_{1}^{m-1} + \gamma z_{2}^{m-1})$ 
of $f(\zz)\overline{g(\zz)}$ where $\gamma \in \Bbb{C}$. 
In \cite{In3}, we take a~coefficient $\gamma$ of $h(\zz)$ such that 
\[
\overline{\gamma} \neq 
\frac{-(2\alpha'f(z, 1) - mg(z, 1))(mz\bar{z}^{m-1}r^{2} + (m-1)\bar{z}^{m-2} - \alpha' z^{m}r^{2})}
{(m-1)(m\bar{z}^{m-1}g(z, 1) - \alpha' f(z, 1))} 
\]
where $(z^{m}+1)\overline{(z+2)} = \alpha'\overline{(z^{m}+1)}(z+2), \alpha' \in S^1$. 
Then $S_{1}(F_{t})$ is the set of indefinite fold singularities and 
the link $S^{3}_{\varepsilon} \cap F_{t}^{-1}(0)$ is a $(m-1, m-1)$-torus link, 
where 
$S^{3}_{\varepsilon} = \{ (z_{1}, z_{2}) \in \Bbb{C}^{2} \mid  \lvert z_{1} \rvert^{2} + \lvert z_{2} \rvert^{2} = \varepsilon \}$, 
$\varepsilon \ll 1$. 
By Lemma $\ref{c3}$, the number of connected components of $S_{1}(F_{t})$ is equal to~$1$. 
\end{Example}

\begin{proof}[Proof of Theorem 1]
By Lemma $\ref{l2}$ and Lemma $\ref{c2}$, 
$D^{4}_{\varepsilon} \cap F_{t}^{-1}(D^{2}_{\delta_{t}})$ consists of a $4$-ball and $\ell$-copies of 
$S^{1} \times B^{3}$ 
and the image of the attaching map $\varphi_{i}$ of $i$-th round $1$-handle is 
contained in both of the boundary of a $4$-ball and that of $S^{1} \times B^{3}$ for $i=1, \dots, \ell$. 
By Lemma $\ref{c3}$, the number of connected components of $S_{1}(F_{t})$ is equal to 
$\lvert L^{-}(P, \mathbf{o})\rvert - \lvert L^{-}(F_{t}, \mathbf{o})\rvert$. 
\end{proof}

\section{Remarks}
\subsection{Monodromy and characteristic polynomials}
Let $h : S \rightarrow S$ be a homeomorphism of a surface $S$. We define 
\[
\Delta_{*}(h) = \frac{\Delta_{1}(h)}{\Delta_{0}(h)}, 
\]
where $\Delta_{i}(h)$ is the characteristic polynomial of the homological map from 
$H_{i}(S, \Bbb{Z})$ to itself induced by $h$ for $i = 0, 1$. 

Let $h_{i} : S_{i} \rightarrow S_{i}$ be the monodromy of $F_{t}\mid_{\tilde{M}_{i}}$ for $i=1, \dots, \ell$. 
Since $h_{i}$ is given by the $S^1$-action on $\Bbb{C}^{2}$, $h_{i} : S_{i} \rightarrow S_{i}$ satisfies the following conditions:  
\renewcommand{\theenumi}{\Roman{enumi}}
\begin{enumerate}
\item
$h_{i}(S_{i-1} \setminus (D'_{1} \sqcup D'_{2})) = S_{i-1} \setminus (D'_{1} \sqcup D'_{2})$ and 
$h_{i}\mid_{S_{i-1}\setminus (D'_{1} \sqcup D'_{2})} = h_{i-1}\mid_{S_{i-1}\setminus (D'_{1} \sqcup D'_{2})}$, 
\item
$h_{i}\mid_{D'_{k}}$ and $h_{i}\mid_{A'}$ are periodic maps which satisfy 
$D^{2}_{k,j} \rightarrow D^{2}_{k,j+1}$ and 
$A_{j} \rightarrow A_{j+1}$ 
\end{enumerate}
for $i = 1, \dots, \ell, j = 1,\dots, d_{p}$ and $k = 1, 2$. Here $D^{2}_{k,d_{p}+1} = D^{2}_{k,1}$ and $A_{d_{p}+1} = A_{1}$. 
We calculate $\Delta_{*}(h_{i})$ from $\Delta_{*}(h_{i-1})$ by using a round $1$-handle $N_i$.

\begin{lemma}\label{l4} 
Let $S_{i}$ be the fiber surface of $F_{t}\mid_{\tilde{M}_{i}}$ and 
$h_{i} : S_{i} \rightarrow S_{i}$ be the monodromy of $F_{t}\mid_{\tilde{M}_{i}}$ for $i = 1, \dots, \ell$. 
Then the characteristic polynomial of $h_i$ satisfies 
\[
\Delta_{*}(h_{i}) = \Delta_{*}(h_{i-1})(t^{d_{p}}-1)^{2}. 
\]
\end{lemma}
\begin{proof}
Since $S_i$ is the surface obtained from $S_{i-1}$ by replacing $D'_{1} \sqcup D'_{2}$ by $A'$ 
and $h_i$ satisfies the above properties, 
we have 
\[
\Delta_{*}(h_{i}) = \frac{\Delta_{*}\bigr(h_{i}\mid_{S_{i-1} \setminus (D'_{1} \sqcup D'_{2})}\bigl)
\Delta_{*}\bigr(h_{i}\mid_{A'}\bigl)}
{\Delta_{*}\bigr(h_{i}\mid_{\partial A'}\bigl)}. 
\]
By the condition $(\text{II})$, the monodromy matrices of $H_{0}(D'_{k}, \Bbb{Z}), H_{i}(A', \Bbb{Z})$ and 
$H_{i}(\partial A', \Bbb{Z})$ are equal to 
\begin{eqnarray*}
 \left(
        \begin{array}{@{\,}cccccccc@{\,}}
0 & 1 & 0 & \ldots & 0 \\
0 & 0 & 1 & \ddots & \vdots \\ 
\vdots & \ddots & \ddots & \ddots & \vdots \\
0 & \ldots & 0 & 0 & 1 \\
1 & 0 & \ldots & 0 & 0
\end{array}
       \right)
\end{eqnarray*}
for $k = 1, 2$ and $i = 0, 1$. 
The characteristic polynomial of the above matrix is equal to $t^{d_{p}} - 1$. 
So $\Delta_{*}\bigr(h_{i}\mid_{A'}\bigl)$ and $\Delta_{*}\bigr(h_{i}\mid_{\partial A'}\bigl)$
are equal to $1$. 
As the condition $(\text{I})$, we have
\[
\Delta_{*}\bigr(h_{i}\mid_{S_{i-1} \setminus (D'_{1} \sqcup D'_{2})}\bigl)
\Delta_{*}\bigr(h_{i-1}\mid_{D'_{1}}\bigl)\Delta_{*}\bigr(h_{i-1}\mid_{D'_{2}}\bigl) = 
\Delta_{*}\bigr(h_{i-1}\bigl).
\] 
Thus the characteristic polynomial satisfies  
\begin{equation*}
\begin{split}
\Delta_{*}(h_{i}) 
&= \Delta_{*}\bigr(h_{i}\mid_{S_{i-1} \setminus (D'_{1} \sqcup D'_{2})}\bigl) \\
&= \Delta_{*}(h_{i-1})(t^{d_{p}} - 1)^{2}. 
\end{split}
\end{equation*}
\end{proof}

Since the two fibrations 
$P : D^{4}_{\varepsilon} \cap F_{t}^{-1}(\partial D^{2}_{\delta}) \rightarrow \partial D^{2}_{\delta}$ and 
$F_{t} : D^{4}_{\varepsilon} \cap F_{0}^{-1}(\partial D^{2}_{\delta}) \rightarrow \partial D^{2}_{\delta}$ 
are isomorphic, we have the following theorem. 

\begin{theorem}
Let $h$ be the monodromy of $P : D^{4}_{\varepsilon} \cap P^{-1}(\partial D^{2}_{\delta}) \rightarrow \partial D^{2}_{\delta}$. 
Then $\Delta_{*}(h)$ is equal to $\Delta_{*}(h_{0})(t^{d_{p}}-1)^{2\ell}$, 
where $\ell$ is the number of connected components of $S_{1}(F_{t})$. 
\end{theorem}
\begin{remark}
The algebraic set $P^{-1}(0) \cap \partial D^{4}_{\varepsilon}$ is a fibered Seifert link in the $3$-sphere. 
Thus the characteristic polynomial of the monodromy of the Milnor fibration of $P$ at the origin 
can also be calculated from the splice diagram \cite{EN}. 
\end{remark}

\subsection{A specific deformation of $f\bar{g}$}
We introduce a deformation of $f\bar{g}$ given in \cite{In3}, where $f$ and $g$ are $2$-variables convenient 
complex polynomials and $f\overline{g}$ has an isolated singularity at the origin $\mathbf{o}$. 
We define the $\Bbb{C}^{*}$-action on $\Bbb{C}^{2}$ as follows: 
\[
c \circ (z_{1}, z_{2}) := (c^{q}z_{1}, c^{p}z_{2}), \ \ c \in \Bbb{C}^{*}. 
\]
Assume that $f(\zz)$ and $g(\zz)$ satisfy
\[ 
f(c \circ \zz) = c^{pqm}f(\zz), \ \ g(c \circ \zz) = c^{pqn}g(\zz), \ \ m > n. 
\]
Then $f(\zz)$ and $g(\zz)$ are weighted homogeneous polynomials. 
Two complex polynomials $f(\zz)$ and $g(\zz)$ can be written as 
\[
 f(\zz) =  \textstyle\prod\limits_{j=1}^{m} (z^{p}_{1} + \alpha_{j}z^{q}_{2}), \ \  
 g(\zz) = \textstyle\prod\limits_{j=1}^{n}(z^{p}_{1} + \beta_{j}z^{q}_{2}), \ \  \gcd(p, q) = 1, 
\]
where $\alpha_{j} \neq \alpha_{j'}, \beta_{j} \neq \beta_{j'} \ (j \neq j')$ and 
$\alpha_{k} \neq \beta_{k'}$ for $1 \leq k \leq m$ and $1 \leq k' \leq n$. 
The mixed polynomial $f(\zz)\overline{g(\zz)}$ is a polar and radial weighted homogeneous polynomial, 
i.e., $f(\zz)\overline{g(\zz)}$ satisfies that 
$f(s \circ \zz)\overline{g(s \circ \zz)} = s^{pq(m - n)}f(\zz)\overline{g(\zz)}$ and 
$f(r \circ \zz)\overline{g(r \circ \zz)} = r^{pq(m + n)}f(\zz)\overline{g(\zz)}$, 
where $s \in S^{1}$ and $r \in \Bbb{R}^{*}$ \cite{O1}. 
We define a deformation of $f(\zz)\overline{g(\zz)}$ as follows: 
\[
F_{t}(\zz) := f(\zz)\overline{g(\zz)} + th(\zz),
\]
where $0 < t \ll 1$ and  
\begin{equation*}
h(\zz) = \begin{cases}
            \gamma_{1}z_{1}^{p(m-n)} + \gamma_{2}z_{2}^{q(m-n)}  & \text{$(g(\zz) \neq z_{1} + \beta z_{2})$}, \\
            z_{1}^{m}\overline{z}_{1} + z_{1}^{m-1} + \gamma z_{2}^{m-1} & \text{$(g(\zz) = z_{1} + \beta z_{2})$}. 
         \end{cases}
\end{equation*} 
Then $F_{t}(\zz)$ is also a polar weighted homogeneous polynomial with the polar degree $pq(m-n)$. 
By \cite[Theorem 1]{In3}, there exists $h(\zz)$ such that $F_{t}(\zz)$ satisfies the conditions $(1), (2)$ and $(3)$ 
for $0 < t \ll 1$. 
The above deformation $F_t$ of $f\bar{g}$ satisfies that $\lvert L^{-}(F_{t}, \mathbf{o})\rvert = 0$. 
By Lemma $\ref{c3}$, the number $\ell$ of connected components of $S_{1}(F_{t})$ is equal to $n$. 
Since the radial degree $d_r$ and the polar degree $d_p$ are equal to $pq(m+n)$ and $pq(m-n)$ respectively, 
we have the following proposition. 
\begin{proposition}
Let $F_t$ be the above deformation of $f\bar{g}$. 
Then the number $\ell$ of connected components of $S_{1}(F_{t})$ is equal to $\frac{d_{r} - d_{p}}{2pq}$, 
where $d_r$ is the radial degree of $f\bar{g}$ and $d_p$ is the polar degree of $f\bar{g}$. 
\end{proposition}

\end{document}